\documentclass{amsart}

\usepackage{amsmath}
\usepackage{amssymb}
\usepackage{paralist}
\usepackage[english]{babel}
\usepackage[colorlinks=true]{hyperref}
\usepackage{enumitem}
\usepackage{array}
\usepackage[latin1]{inputenc}
\usepackage{esint}
\usepackage{latexsym,amsfonts}
\usepackage{amsthm}
\usepackage{cleveref}
\usepackage{mathtools}
\usepackage{verbatim}
\usepackage{graphics}

\numberwithin{equation}{section}

\newtheorem{theorem}{Theorem}
\newtheorem{lemma}{Lemma}[section]
\newtheorem{proposition}{Proposition}[section]
\newtheorem{corollary}{Corollary}[section]

\theoremstyle{definition}
\newtheorem{definition}{Definition}[section]
\newtheorem*{remark}{Remark}

\newcommand{\norm}[1]{\left\|#1\right\|} 
\newcommand{\pare}[1]{\left(#1\right)} 
\newcommand{\brackets}[1]{\left[#1\right]} 
\newcommand{\set}[2]{\left\{#1 \; :\; #2\right\}} 

\DeclareMathOperator*{\diver}{div} 
\DeclareMathOperator*{\diam}{diam} 
\DeclareMathOperator*{\dist}{dist} 

\newcommand{\R}{\mathbb R} 
\newcommand{\N}{\mathbb N} 
\newcommand{\T}{\mathcal{T}}
\newcommand{\M}{\mathcal{M}}
\renewcommand{\S}{\mathcal{S}}
\renewcommand{\L}{\mathcal{L}}

\renewcommand{\geq}{\geqslant}
\renewcommand{\leq}{\leqslant}

\def\Xint#1{\mathchoice
   {\XXint\displaystyle\textstyle{#1}}%
   {\XXint\textstyle\scriptstyle{#1}}%
   {\XXint\scriptstyle\scriptscriptstyle{#1}}%
   {\XXint\scriptscriptstyle\scriptscriptstyle{#1}}%
   \!\int}
\def\XXint#1#2#3{{\setbox0=\hbox{$#1{#2#3}{\int}$}
     \vcenter{\hbox{$#2#3$}}\kern-.5\wd0}}

\def\dashint{\Xint-}

\begin{document}

\title{$p$-harmonic functions by way of intrinsic mean value properties}

\author[Arroyo]{\'Angel Arroyo}
\address{MaLGa Center, Department of Mathematics, University of Genoa, Via Dodecaneso 35, 16146 Genova, Italy}
\email{arroyo@dima.unige.it}

\author[Llorente]{Jos\'e G. Llorente}
\address{Departament de Matem\`{a}tiques, Universitat Aut\`onoma de Barcelona, 08193 Bellaterra, Barcelona, Spain}
\email{jgllorente@mat.uab.cat}



\thanks{Partially supported by grants MTM2017-85666-P, 2017 SGR 395. Part of this work has been carried out at the Machine Learning Genoa (MaLGa) center, Universit\'a di Genova (IT). \'A. A. is supported by a UniGe starting grant ``curiosity driven''.}

\begin{abstract}
Let $\Omega\subset\mathbb{R}^n$ be a bounded domain satisfying the
uniform exterior cone condition. We establish existence and
uniqueness of continuous solutions of the Dirichlet Problem
associated to certain intrinsic nonlinear mean value properties in $
\Omega $. Furthermore we show that, when properly normalized, such
functions converge to the $p$-harmonic solution of the Dirichlet
problem in $\Omega$, for $p\in[2,\infty)$. The proof of existence is
constructive and the methods are entirely analytic, a fundamental
tool being the construction of explicit, $p$-independent barrier
functions in $\Omega$.
\end{abstract}

\maketitle

\section{Introduction}
\subsection{Background}

Let $\Omega\subset\R^n$ be a bounded domain. The aim of this paper
is twofold. First we discuss existence and uniqueness of continuous
solutions of the Dirichlet Problem in $\Omega$ associated to certain
nonlinear mean value properties in balls of variable radius $B(x,
\rho (x)) \subset \Omega$. Then we study the convergence of such
functions to the solution of the $p$-harmonic Dirichlet problem in
$\Omega$.

\

A primary motivation of our work is the  mean value property for
harmonic functions, saying that a continuous function $u$ in a
domain $\Omega \subset \R^n$ is harmonic if and only if
\begin{equation}\label {basic mvp}
u(x) = \dashint_{B(x,\rho )} \hspace{-10pt} u(y) \ dy
\end{equation}
for each $x\in \Omega$ and each $\rho >0$ such that $ 0 < \rho <
\dist (x, \partial \Omega )$. The mean value property plays a
relevant role in Geometric Function Theory and is indeed the
fundamental tool of the interplay between Classical Potential
Theory, Probability and Brownian motion.

\

A theorem due to Volterra (for regular domains) and Kellogg (in the
general case) says that if $\Omega\subset \R^n$ is bounded, $u\in
C(\overline{\Omega})$ and for each $x\in \Omega$ there is a radius
$\rho = \rho (x)$ with $0< \rho \leq \dist (x,
\partial \Omega )$ such that \eqref{basic mvp} holds then $u$ is harmonic in $\Omega$ (\cite{V,K}). Therefore, under
appropriate hypothesis, the mean value property for a single radius
(depending on the point) implies harmonicity.

\

In the last decades, substantial efforts have been devoted to
determine the stochastic structure of certain nonlinear PDE's, a
crucial step being the identification of the corresponding
(nonlinear) mean value properties. In this paper we will focus on
the $p$-\emph{laplacian} which, for $1 < p < \infty$, is the
divergence-form differential operator given by
\begin{equation*}
\Delta_p u  :\,= \diver ( |\nabla u |^{p-2} \nabla u ).
\end{equation*}
Weak solutions $ u \in W_{loc}^{1,p}(\Omega )$ of $\Delta_p u = 0$
are said $p$-\emph{harmonic} functions. Observe that the theory is
nonlinear unless $p =2$, in which case we recover harmonic
functions. We refer to \cite{L} for background and basic properties
of $p$-harmonic functions.

\

Unfortunately, the nature of the connections between $p$-harmonic
functions and mean value properties is more delicate when $p\neq 2$.
We start by some basic facts in the smooth case. If $u\in C^2$ and
$\nabla u \neq 0$ then a direct computation gives
\begin{equation}\label{p-lapla}
    \Delta_p u
    =
    |\nabla u |^{p-2} \pare{ \Delta u +
(p-2)\frac{\Delta_{\infty}u}{|\nabla u |^2} },
\end{equation}
where
$$
\Delta_{\infty}u = \sum_{i,j=1}^n u_{x_i}u_{x_j}u_{x_i , x_j}
$$
is the so called \emph{$\infty$-laplacian} in $\R^n$. Then
\eqref{p-lapla} shows that, in the smooth case and away from the
critical points, the $p$-laplacian can be understood as a linear
combination of the usual laplacian and the normalized
$\infty$-laplacian.

\

By using the viscosity characterization of $p$-harmonic functions
(\cite{JUU-LIN-MAN}), Manfredi, Parviainen and Rossi characterized
$p$-harmonicity in terms of nonlinear mean value properties in
\cite{MAN-PAR-ROS-10}. Namely, a function $u\in C(\Omega)$ is
$p$-harmonic in $\Omega\subset\R^n$ if and only if $u$ satisfies the
\emph{asymptotic $p$-mean value property}
\begin{equation} \label{p-mvp}
    u(x)
    =
    \frac{p-2}{n+p}\,\pare{\frac{1}{2}\,\sup_{B(x,\varepsilon)}u + \frac{1}{2}\,\inf_{B(x,\varepsilon)}u}
    +
    \frac{n+2}{n+p}\,\dashint_{B(x,\varepsilon)}\hspace{-10pt}u(y)\ dy+o(\varepsilon^2)
    \quad
\end{equation}
in a viscosity sense for each $x\in\Omega$. If $n=2$ this
characterization holds also in the classical sense
(\cite{LIN-MAN,ARR-LLO-16-2}), while for $n\geq 3$ the question of
whether  $p$-harmonic functions satisfy \eqref{p-mvp} in the
classical sense is still open.
Note that if $p=2$ then \eqref{p-mvp} is actually equivalent to $u$ being harmonic.

\

From a probabilistic point of view, the influential work of Peres,
Schramm, Sheffield and Wilson (\cite{PER-SCH-SHE-WIL}) established a
game-theoretic interpretation of the $\infty$-laplacian and the
functional equation
\begin{equation*}
    u_\varepsilon(x)
    =
    \frac{1}{2}\,\sup_{B(x,\varepsilon)}u_\varepsilon + \frac{1}{2}\,\inf_{B(x,\varepsilon)}u_\varepsilon
\end{equation*}
appears as a \emph{dynamic programming principle} of a two-player
zero-sum tug-of-war game. A similar interpretation for the
$p$-laplacian, $p\in[2,\infty]$, was considered in \cite{PER-SHE}.
Manfredi, Parviainen and Rossi gave a systematic twist to the
theory, from both an analytic and probabilistic point of view
(\cite{MAN-PAR-ROS-10, MAN-PAR-ROS-12}). In particular, in
\cite{MAN-PAR-ROS-12} the term $p$-\emph{harmonious} was
introduced to denote (not necessarily continuous) solutions of the
functional equation
\begin{equation}\label{p-rmvp}
    u_\varepsilon(x)
    =
    \frac{p-2}{n+p}\,\pare{\frac{1}{2}\,\sup_{B(x,\varepsilon)}u_\varepsilon + \frac{1}{2}\,\inf_{B(x,\varepsilon)}u_\varepsilon}
    +
    \frac{n+2}{n+p}\,\dashint_{B(x,\varepsilon)}\hspace{-10pt}u_\varepsilon(y)\ dy
\end{equation}
for each $x\in\Omega$. Note, however, that \eqref{p-rmvp} raises some
technical problems, coming from the fact that the balls
$B(x,\varepsilon)$ eventually escape the domain. The authors in
\cite{MAN-PAR-ROS-12} extended a given $f\in C(\partial \Omega ) $
to the strip $\displaystyle \{ x\in \R^n \setminus \Omega \, :
\dist(x,
\partial \Omega ) \leq \varepsilon \} $ and proved that, if $\Omega
\subset \R^n$ is bounded and satisfies a so called \emph{boundary
regularity condition}, then there is a unique $p$-harmonious
function $u_{\varepsilon}$ having $f$ as boundary values (in the
extended sense). Furthermore, $u_{\varepsilon}\to u$ uniformly in
$\overline{\Omega}$ as $\varepsilon \to 0$, where $u$ is the unique
$p$-harmonic function solving the Dirichlet Problem in $\Omega$ with
boundary data $f$. It should be remarked that domains satisfying a
uniform exterior cone condition (see \Cref{DEF-ECC} below) verify
the boundary regularity condition, in the sense of
\cite{MAN-PAR-ROS-12}. See also \cite{LUI-PAR-SAK, ARR-HEI-PAR,
DEL-MAN-PAR} for further approaches.

\subsection{Main results}

In this paper we deal with a modified version of \eqref{p-rmvp} in
which the balls $B(x,\varepsilon)$ are replaced by balls of variable
radius $B(x,\rho(x))$, where $0<\rho(x)\leq\dist(x,\partial\Omega)$.
We want to emphasize that the variable radius setting is natural for
at least two reasons: it is closely related to the classical theory
(remind Volterra-Kellogg theorem) and it is intrinsic, in the sense
that no extension of the domain is needed (this explains the term
\emph{intrinsic} in the title).

\

\Cref{MAIN-THM-EXISTENCE} below is an existence and uniqueness
result for the Dirichlet Problem associated to intrinsic mean value
properties. It extends the existence result in \cite{MAN-PAR-ROS-12}
to the variable radius setting, substantially relaxes  the
geometrical restrictions  of \cite{ARR-LLO-16-1} and, as an
additional feature, the solution is constructively obtained via
iteration of the averaging operators $\T_{\rho,p}$ (see
\eqref{Trho}).

\

\Cref{MAIN-THM-CONVERGENCE} is an approximation result showing
that, when properly normalized, solutions of intrinsic mean value
properties with fixed continuous boundary data converge to the
solution of the $p$-harmonic Dirichlet Problem with the same
boundary data. The combination of \Cref{MAIN-THM-EXISTENCE} and
\Cref{MAIN-THM-CONVERGENCE} provides therefore an intrinsic and
constructive method of obtaining solutions of $p$-harmonic Dirichlet
problems which might be of interest from a computational point of
view.

\

The fundamental tool to prove \Cref{MAIN-THM-EXISTENCE} and
\Cref{MAIN-THM-CONVERGENCE} is the construction of explicit barriers
which do not depend on $p$ and work simultaneously for the operators
$\T_{\rho,p}$ and the $p$-laplacian (\Cref{mainthmloc}). We believe
that this construction has an independent interest which might be
useful in other situations.

\

Even though the motivation for studying such problems is partly
probabilistic (namely the search of ``natural'' stochastic processes
associated to the $p$-laplacian), our arguments and techniques are
entirely analytic.

\

Before stating the main results, let us introduce some necessary
definitions.

\begin{definition}
Let $\Omega\subset\R^n$ be a bounded domain. We say that a function
$\rho:\overline\Omega \to(0,\infty)$ is an \emph{admissible radius
function} in $\Omega$ if
\begin{enumerate}
\item $0<\rho(x)\leq\dist(x,\partial\Omega)$ for every $x\in\Omega$, and
\item $\rho(x)=0$ if and only if $x\in\partial\Omega$.
\end{enumerate}
Hereafter we will write  $B_\rho(x):\,=B(x,\rho(x))$.
\end{definition}

\begin{definition}
Let $\Omega\subset\R^n$ be a bounded domain and $\rho$ an admissible
radius function in $\Omega$. Let $\S_\rho$ and $\M_\rho$ be the
operators in $L^\infty(\overline\Omega)$ defined by
\begin{equation*}
    \S_\rho u(x)
    :\,=
    \begin{cases}
        \displaystyle \frac{1}{2}\sup_{B_\rho(x)}u+\frac{1}{2}\inf_{B_\rho(x)}u
        & \text{ if } x\in\Omega,
        \\
        u(x)
        & \text{ if } x\in\partial\Omega,
    \end{cases}
\end{equation*}
and
\begin{equation*}
    \M_\rho u(x)
    :\,=
    \begin{cases}
        \displaystyle \dashint_{B_\rho(x)}u(y)\ dy
        & \text{ if } x\in\Omega,
        \\
        u(x)
        & \text{ if } x\in\partial\Omega,
    \end{cases}
\end{equation*}
for every $u\in L^\infty(\overline\Omega)$. In addition, for a fixed
$p\in[2,\infty)$, we define the operator $\T_{\rho,p}$ in
$L^\infty(\overline\Omega)$ as the following linear combination of
$\S_\rho$ and $\M_\rho$
\begin{equation}\label{Trho}
    \T_{\rho,p}
    :\,=
    \frac{p-2}{n+p}\,\S_\rho+\frac{n+2}{n+p}\,\M_\rho.
\end{equation}
\end{definition}

Note that $\T_{\rho , 2} = \M_{\rho}$ and that, formally, $\T_{\rho
, \infty}=\S_{\rho}$. As the following proposition says,
$\T_{\rho,p}$ preserves the class $C(\overline\Omega)$, provided the
admissible radius function $\rho$ is continuous.
\begin{proposition}[{\cite[Proposition 4.1]{ARR-LLO-18}}]\label{PROP-TRHO-CONT}
If $\rho\in C(\overline\Omega)$, then
$\T_{\rho,p}:C(\overline\Omega)\to C(\overline\Omega)$.
\end{proposition}
From now on we will take $C(\overline\Omega)$ as the natural
function space where the operators $\T_{\rho,p}$ are defined. In
addition, $\T_{\rho,p}$ satisfies the following properties:
\begin{enumerate}
\item \emph{Affine invariance}: if $a,b\in\R$ and $u\in C(\overline\Omega)$, then $\T_{\rho,p}\big(au+b)=a\T_{\rho,p} u+b$.
\item \emph{Monotonicity}: if $u,v\in C(\overline\Omega)$ such that $u\leq v$, then $\T_{\rho,p} u\leq\T_{\rho,p} v$.
\item \emph{Non-expansiveness}: if $u,v\in C(\overline\Omega)$, then $\|\T_{\rho,p} u-\T_{\rho,p} v\|_\infty\leq\|u-v\|_\infty$.
\item $\inf_{B_\rho(x)}u\leq\T_{\rho,p} u(x)\leq\sup_{B_\rho(x)}u$ for every $x\in\Omega$.
\end{enumerate}
It is easy to check that the $k$-th iteration of $\T_{\rho,p}$,
denoted by $\T_{\rho,p}^k$, also satisfies the above four
properties.

\

Let  $f\in C(\partial\Omega)$. In this paper we are interested in
existence and uniqueness of solutions of the Dirichlet Problem
\begin{equation}\label{dirTrho}
\begin{cases} \T_{\rho,p}u   & =  \, u \, \, \, \,  \text{in} \, \, \, \Omega, \\
u  & =  \, f \, \, \, \,  \text{on} \, \, \,
\partial \Omega.
\end{cases}
\end{equation}
We will also discuss assumptions under which normalized solutions of
\eqref{dirTrho} converge to the corresponding solution of the
Dirichlet Problem for the $p$-laplacian.

\

Notice that \eqref{dirTrho} is equivalent to the problem of finding
a fixed point of $\T_{\rho,p}$ among all continuous functions with
prescribed continuous boundary data $f$. Given $f\in
C(\partial\Omega)$, we define $\mathcal{K}_f$ as the set of all
\emph{norm-preserving continuous extensions of $f$ to
$\overline\Omega$},
\begin{equation}
\label{Kf} \mathcal{K}_f :\,=
                \set{u\in C(\overline\Omega)}{u\big|_{\partial\Omega}=f
                \text{ and } \|u\|_{\infty,\Omega}=\|f\|_{\infty,\partial\Omega}}.
\end{equation}
By \Cref{PROP-TRHO-CONT} and the non-expansiveness of the operator,
it follows that $\T_{\rho,p}(\mathcal{K}_f)\subset \mathcal{K}_f$.
Furthermore, if $u$ satisfies \eqref{dirTrho} then  $
\|u\|_{\infty,\Omega} = \|f\|_{\infty,\partial\Omega}$ by the
Comparison Principle (\Cref{THM-COMPARISON-PPLE}).
Therefore the Dirichlet Problem \eqref{dirTrho} has a solution in
$C(\overline{\Omega})$ if and only if $\T_{\rho,p}$ has a fixed
point in $\mathcal{K}_f$.

\

In order to state the main theorems of this work, we need to impose
certain geometrical condition on the boundary of the domain.

\begin{definition}[Uniform Exterior Cone Condition]\label{DEF-ECC}
Let $\alpha\in(0,\frac{\pi}{2})$ and $r>0$. We denote by
$K_{\alpha,r}$ the truncated circular cone
\begin{equation*}
    K_{\alpha,r}
    =
    \set{x\in\R^n}{x_1\leq-|x|\cos\alpha \ \text{ and } \ |x|\leq r}.
\end{equation*}
We say that a domain $\Omega\subset\R^n$ satisfies the \emph{uniform
exterior cone condition} if there exist constants
$\alpha\in(0,\frac{\pi}{2})$ and $r>0$ such that for every
$\xi\in\partial\Omega$ there is a rotation $R\in SO(n)$ in $\R^n$
such that
\begin{equation*}
                \xi+R(K_{\alpha,r})
                \subset
                \R^n\setminus\Omega.
\end{equation*}
\end{definition}

\begin{remark}
Let $\Omega\subset\R^n$ be a bounded domain. Then $\Omega$ is
Lipschitz if and only if both $\Omega$ and $\R^n\setminus\Omega$
satisfy the uniform exterior cone condition (see \cite{GRI}).
\end{remark}

We now state the first main result of this paper (compare with
\cite{ARR-LLO-16-1} where the same result was proven under the
assumption that $\Omega$ is strictly convex and the admissible
radius function is $1$-Lipschitz). See \cite{MAN-PAR-ROS-12,
LUI-PAR-SAK} for previous versions in the constant radius case.

\begin{theorem}\label{MAIN-THM-EXISTENCE}
Let $\Omega\subset\R^n$ be a domain satisfying the uniform exterior
cone condition and $p \in [2, \infty)$. Suppose that $\rho\in
C(\overline\Omega)$ is a continuous admissible radius function in
$\Omega$ satisfying
\begin{equation*}
                \lambda\dist(x,\partial\Omega)^\beta
                \leq
                \rho(x)
                \leq
                \Lambda\dist(x,\partial\Omega)
\end{equation*}
for all $x\in\Omega$, where
\begin{equation}\label{ctts}
                \beta
                \geq
                1,
                \quad
                0
                <
                \Lambda
                <
                1-\pare{\frac{p-2}{n+p}}^{1/\beta}
                \quad\text{ and }\quad
                0
                <
                \lambda
                \leq
                \Lambda \pare{\frac{\diam\Omega}{2}}^{1-\beta}.
\end{equation}
Then for any $f\in C(\partial\Omega)$ there exists a unique solution
$u_\rho\in C(\overline\Omega) $ to the Dirichlet Problem
\begin{equation}\label{DP}
                \begin{cases}
                \T_{\rho,p} u_\rho=u_\rho
                & \mbox{ in } \Omega,
                \\
                u_\rho=f
                & \text{ on } \partial\Omega,
                \end{cases}
\end{equation}
where $\T_{\rho,p}$ is the averaging operator defined in
\eqref{Trho}. Furthermore, for any norm-preserving continuous
extension $u\in C(\overline\Omega)$ of $f$, the sequence of iterates
$\{\T_{\rho,p}^ku\}_k$ converges uniformly to $u_\rho$ in
$\overline\Omega$.
\end{theorem}

By letting $\beta=1$, we obtain the following corollary as an
immediate consequence.

\begin{corollary}\label{COR-EXISTENCE}
Let $\Omega\subset\R^n$ be a domain satisfying the uniform exterior
cone condition and $p\in [2, \infty)$. Suppose that $\rho\in
C(\overline\Omega)$ is a continuous admissible radius function in
$\Omega$ satisfying
\begin{equation*}
    \lambda\dist(x,\partial\Omega)
    \leq
    \rho(x)
    \leq
    \Lambda\dist(x,\partial\Omega)
\end{equation*}
for all $x\in\Omega$, where
\begin{equation*}
    0<\lambda\leq\Lambda<\frac{n+2}{n+p}.
\end{equation*}
Then for any $f\in C(\partial\Omega)$ there exists a unique solution
$u_\rho\in C(\overline\Omega)$ to the Dirichlet Problem \eqref{DP}.
Furthermore, for any norm-preserving continuous extension $u\in
C(\overline\Omega)$ of $f$, the sequence of iterates
$\{\T_{\rho,p}^ku\}_k$ converges uniformly to $u_\rho$ in
$\overline\Omega$.
\end{corollary}

\

The next theorem is our second main result of the paper. It says
that, when considering a family of admissible radius functions going
to zero in an appropriate way, the corresponding solutions given by
\Cref{MAIN-THM-EXISTENCE} converge uniformly to the $p$-harmonic
solution of the Dirichlet Problem.

\begin{theorem}\label{MAIN-THM-CONVERGENCE}
Let $\Omega\subset\R^n$ be a domain satisfying the uniform exterior
cone condition and $p\in [2, \infty)$. Suppose that
$\{\rho_\varepsilon\}_{0<\varepsilon\leq 1}$ is a collection of
continuous admissible radius functions in $\Omega$ satisfying
\begin{equation*}
                \lambda\dist(x,\partial\Omega)^\beta
                \leq
                \frac{\rho_\varepsilon(x)}{\varepsilon}
                \leq
                \Lambda\dist(x,\partial\Omega)
\end{equation*}
for all $x\in\Omega$ and every $0<\varepsilon\leq 1$, where $\beta$,
$\lambda$ and $\Lambda$ are as in \eqref{ctts}. Given any continuous
boundary data $f\in C(\partial\Omega)$, let $u_{\varepsilon}$ be the
solution of
\begin{equation*}
                \begin{cases}
                \T_{\rho_\varepsilon,p} u_\varepsilon=u_\varepsilon
                & \mbox{ in } \Omega,
                \\
                u_\varepsilon=f
                & \text{ on } \partial\Omega.
                \end{cases}
\end{equation*}
Then $u_\varepsilon\to u_0$ uniformly in $\overline\Omega$, where
$u_0$ is the unique $p$-harmonic function in $\Omega$ solving
\begin{equation*}
                \begin{cases}
                \Delta_p u_0=0 & \mbox{ in } \Omega,
                \\
                u_0=f & \mbox{ on } \partial\Omega.
                \end{cases}
\end{equation*}
\end{theorem}

\

The fundamental tool for the proofs of \Cref{MAIN-THM-EXISTENCE} and
\Cref{MAIN-THM-CONVERGENCE} is provided by \Cref{mainthmloc} below.

\begin{definition}
Let $\Omega\subset\R^n$ be a domain and $\xi\in\partial\Omega$. We
say that a function $w_\xi\in C(\overline\Omega)$ is a
\emph{$\T_{\rho,p}$-barrier} at $\xi$ if $w_\xi>0$ in
$\overline\Omega\setminus\{\xi\}$, $w_\xi(\xi)=0$,  and
$w_\xi\geq\T_{\rho,p} w_\xi$ in $\Omega$. In this case, we say that
$\xi$ is a \emph{$\T_{\rho,p}$-regular} point. Moreover, a bounded
domain $\Omega\subset\R^n$ is \emph{$\T_{\rho,p}$-regular} if every
point on $\partial\Omega$ is $\T_{\rho,p}$-regular.
\end{definition}

\begin{theorem}\label{mainthmloc}
Let $\Omega \subset \R^n$ be a bounded domain satisfying the uniform
exterior cone condition with constants $\alpha \in (0, \frac{\pi}{2}
)$ and $r>0$ as in \Cref{DEF-ECC}. Choose $\gamma$ such that
\begin{equation}\label{gamma}
0 < \gamma < \frac{8(\sin \alpha )^{n-2}}{n(13 \pi^2 + 4\pi )}.
\end{equation}
Then for each $\xi \in \partial \Omega$ there exists a function
$w_{\xi} \in C(\overline{\Omega})$ such that $w_{\xi}(\xi ) = 0$,
$w_{\xi} >0 $ in $\overline{\Omega} \setminus \{ \xi \}$,
\begin{equation}\label{MS-supersolution}
    w_{\xi}(x)
    \geq
    \dashint_{B(x,\varrho)}w_{\xi}(y)\ dy
    \quad , \qquad
    w_{\xi}(x)
    \geq
    \frac{1}{2}\sup_{B(x,\varrho)}w_{\xi}+\frac{1}{2}\inf_{B(x,\varrho)}w_{\xi}
\end{equation}
for every ball $B(x,\varrho)\subset\Omega$ and
\begin{equation}\label{lower-upper-control}
\L (|x-\xi|) \leq w_{\xi}(x) \leq \gamma^{-2}|x-\xi|^{\gamma}
\end{equation}
for every $x\in \overline{\Omega}$, where $\L (t) :=
\alpha^{2-n}\min \{t,r\}^{\gamma} $. In particular, for each $ p \in
[2, \infty]$ and any admissible radius function $\rho$ in $\Omega$,
$w_{\xi}$ is, simultaneously, a $\T_{\rho ,p}$-barrier and a barrier
for the $p$-laplacian  at $\xi$ in $\Omega$.
\end{theorem}

\

\subsection{Furher remarks}

It is worth to recall that the Dirichlet Problem \eqref{dirTrho} for
$p=\infty$ was studied by Le Gruyer and Archer in \cite{LEG-ARC} in
the context of metric spaces. There, functions satisfying $\S_\rho
u=u$ were originally called \emph{harmonious} and studied in
connection to extension problems of continuous functions in metric
spaces. It follows, as a particular case of results in
\cite{LEG-ARC}, that if $\Omega \subset \R^n $ is a bounded and
convex domain and $\rho$ is $1$-Lipschitz then the Dirichlet Problem
\begin{equation}\label{dirS}
\begin{cases} \S_{\rho}u   & =  \, u \, \, \, \,  \text{in} \, \, \, \Omega, \\
u  & =  \, f \, \, \, \,  \text{on} \, \, \,
\partial \Omega,
\end{cases}
\end{equation}
has a unique solution for each $f\in C(\overline{\Omega})$. One of
the most important features of the operator $\S_\rho$ (with
$1$-Lipschitz $\rho$) is that it preserves the concave modulus of
continuity: if $\widehat{\omega}_u$ is the lowest concave modulus of
continuity of $u$ in $\Omega$, then $\widehat{\omega}_{\S_\rho
u}\leq\widehat{\omega}_u$. This invariance property allows the use
of Schauder's fixed point theorem to prove existence in
\eqref{dirS}. Unfortunately, the operators $\T_{\rho,p}$ do not
preserve in general  the modulus of continuity, thus Schauder's
Theorem is no longer available and different strategies are required
to obtain fixed points.

\

The existence part in \Cref{MAIN-THM-EXISTENCE} is obtained from the
equicontinuity and subsequent uniform convergence in
$\overline{\Omega}$ of the iterates $\{ \T_{\rho,p}^k u \}_k$, for
$u\in C (\overline{\Omega})$. When $p = 2$, the fact that the
sequence $\{ \M_{\rho}^k u \}_k $ converges uniformly in
$\overline{\Omega}$ to the solution of the (harmonic) Dirichlet
Problem in $\Omega$ with boundary data $f = u|_{\partial \Omega}$
was already observed by Lebesgue, in the case that $\Omega$ is
regular and $\rho (x) = \dist(x, \partial \Omega )$ (\cite{Le}, see
also \cite{C} for a more general approach in this direction). A
significative difference between Lebesgue's setting and the methods
of this paper is that in Lebesgue's note existence is taken for
granted and the convergence of the iterates is obtained as a
consequence, while we actually use the convergence of the iterates
to prove existence. As for equicontinuity, it is worth mentioning
that boundary equicontinuity turns out to be a much more delicate
matter than interior equicontinuity. In \cite{ARR-LLO-16-1},
boundary equicontinuity was established under the assumption that
$\Omega$ is strictly convex and $\rho$ is $1$-Lipschitz. Our
approach here is based on the construction of explicit barriers for
$\T_{\rho,p}$, having the additional advantage that they work for
domains satisfying the uniform exterior cone condition. We would
like to point out that, since the operators $\T_{\rho,p}$ are not
local, some steps in Perron's method (like Poisson's modification)
do not work in our setting and we need \emph{ad hoc} arguments to
prove existence. Our approach gives, in particular, a more
constructive proof of the existence of solution to the Dirichlet
problem for the $p$-laplacian in domains satisfying a uniform
exterior cone condition.

\

The rest of the paper is organized as follows: in \cref{barr} we
show the existence of $\T_{\rho,p}$-barriers for domains satisfying
the uniform exterior cone condition (\Cref{mainthmloc}). Then in
\cref{sec.existence,sec.convergence} we use these barriers to prove
existence of fixed points of $\T_{\rho,p}$
(\Cref{MAIN-THM-EXISTENCE}) and their convergence to $p$-harmonic
functions (\Cref{MAIN-THM-CONVERGENCE}), respectively. For the sake
of convenience and, whenever the role of $p\in[2,\infty)$ causes no
confusion, we will write $\T_\rho$ instead of $\T_{\rho,p}$ in what
follows.

\

\noindent\textbf{Acknowledgements.} We wish to thank F. del Teso,
J.J. Manfredi and M. Parviainen for bringing to our attention their
preprint \cite{DEL-MAN-PAR}, which motivated part of  this work.

\section{Barriers for $\T_\rho$} \label{barr}

Our goal is to construct a $\T_\rho$-barrier
at each boundary point
and consequently, to show that each point on the boundary is
$\T_\rho$-regular. Fix $\xi\in\partial\Omega$. Recalling the
definition of the uniform exterior cone condition, there exist
constants $\alpha\in(0,\frac{\pi}{2})$ and $r>0$ and a rotation
$R_\xi\in SO(n)$ such that
\begin{equation*}
    \xi+R_\xi(K_{\alpha,r})
    \subset
    \R^n\setminus\Omega,
\end{equation*}
where
\begin{equation}\label{Kr}
    K_{\alpha,r}
    =
    \set{x\in\R^n}{x_1\leq-|x|\cos\alpha \ \text{ and } \ |x|\leq r}.
\end{equation}
We observe that after a translation and a rotation, we can assume
that $\xi=0$ and $R_\xi=\mathrm{Id}$, in which case we define a
bigger domain $\Omega_{\alpha,r}=\R^n\setminus K_{\alpha,r}$ so that
$\Omega\subset\Omega_{\alpha,r}$. Then our aim is to construct a
function $w$ in $\Omega_{\alpha,r}$ such that its restriction to
$\Omega$, $w\big|_\Omega$, verifies $w \geq \T_\rho w $ for every
admissible radius function $\rho$ in $\Omega$.

We split the construction of such function in two steps. First, we
construct the barrier at $0$ for the complement of an unbounded cone
along the negative $x_1$-axis. Second, we adapt the argument to work for the
complement of a truncated cone.

\subsection{Barrier for the complement of a whole cone}

Let $\alpha\in(0,\frac{\pi}{2})$ and define
\begin{equation*}
    \Omega_\alpha
    :\,=
    \set{x\in\R^n}{x_1>-|x|\cos\alpha}.
\end{equation*}
We will use \emph{polar} coordinates with respect to the $x_1$-axis, that is, we assign a pair $(R,\theta)$ to each $x\in\R^n$,
where $R=|x|$ and $\theta=\arccos\big(\frac{x_1}{|x|}\big)\in[0,\pi)$ is the angle between $x$ and the positive $x_1$-axis. Then,
\begin{equation*}
    \Omega_\alpha
    =
    \set{x\in\R^n}{0\leq\theta<\pi-\alpha}.
\end{equation*}

Before stating  the main result of this section, we define an auxiliary function
$\phi:(-\pi,\pi)\to[0,\infty)$ as the solution of the differential
equation
\begin{equation}\label{edo}
    \begin{cases}
        \phi''(\theta)+(n-2)\phi'(\theta)\cot\theta
        =
        1,
        \\
        \phi(0)=\phi'(0)=0,
    \end{cases}
\end{equation}
which has the integral form
\begin{equation}\label{defphi}
    \phi(\theta)
    =
    \int_0^{|\theta|}\int_0^t\bigg(\frac{\sin s}{\sin t}\bigg)^{n-2}\ ds\ dt
\end{equation}
for every $\theta\in(-\pi,\pi)$ (see
\cite[Lemma 2.4]{HAY-KEN}). We review some of the properties of the
auxiliary function $\phi$ in the following lemma.

\begin{lemma}\label{phi}
The function $\phi:(-\pi,\pi)\to[0,\infty)$ defined in \eqref{defphi} satisfies:
\begin{enumerate}[label=\roman*)]
\item $\phi\in C^2(-\pi,\pi)$.
\item $\phi$ is increasing in $(0,\pi)$ and convex in $(-\pi,\pi)$.
\item For every $|\theta|\leq\pi-\alpha$,
\begin{equation}\label{bounds-phi}
    0
    \leq
    \phi(\theta)
    \leq
    \frac{\pi^2}{8}+\frac{\pi}{2(\sin\alpha)^{n-2}}
    \qquad \text{ and } \qquad
    \phi'(\theta)
    \leq
    \frac{\pi}{(\sin\alpha)^{n-2}}.
\end{equation}
\end{enumerate}
\end{lemma}

\begin{proof}
It is easy to check that $\phi\in C^2(-\pi,\pi)$. Hereafter, we restrict the analysis to the interval $[0,\pi)$. By differentiation of \eqref{defphi},
\begin{equation*}
    \phi'(\theta)
    =
    \frac{1}{(\sin\theta)^{n-2}}\int_0^\theta(\sin t)^{n-2}\ dt
    \geq
    0,
\end{equation*}
so $\phi$ is increasing in $(0,\pi)$. Next, since $\phi$ satisfies \eqref{edo}, then
\begin{equation*}
    \phi''(\theta)
    =
    \frac{(\sin\theta)^{n-1}-(n-2)\cos\theta\displaystyle\int_0^\theta(\sin t)^{n-2}\ dt}{(\sin\theta)^{n-1}}
\end{equation*}
for $0<\theta<\pi$. Observe that if $\frac{\pi}{2}\leq\theta<\pi$, then $\cos\theta\leq 0$, so $\phi''\geq 0$ in $[\frac{\pi}{2},\pi)$. For $0\leq\theta\leq\frac{\pi}{2}$ define
\begin{equation*}
    \psi(\theta)
    =
    (\sin\theta)^{n-1}-(n-2)\cos\theta\int_0^\theta(\sin t)^{n-2}\ dt
\end{equation*}
and observe that $\psi(0)=0$ and
\begin{equation*}
    \psi'(\theta)
    =
    \cos\theta(\sin\theta)^{n-2} + (n-2)\sin\theta\int_0^\theta(\sin t)^{n-2}\ dt
    \geq
    0
\end{equation*}
for $0\leq\theta\leq\frac{\pi}{2}$. Therefore $\phi$ is convex in $(0,\pi)$.

To show \eqref{bounds-phi} note first that
\begin{equation*}
    \int_0^\theta(\sin t)^{n-2}\ dt
    \leq
    \theta\max_{0\leq t\leq\theta}\{(\sin t)^{n-2}\},
\end{equation*}
for each $0<\theta<\pi$, so
\begin{equation}\label{phi-d1}
    \phi'(\theta)
    \leq
    \begin{cases}
        \theta & \displaystyle \text{ if } 0\leq\theta\leq\frac{\pi}{2},
        \\
        \displaystyle\frac{\theta}{(\sin\theta)^{n-2}} & \text{ if } \displaystyle\frac{\pi}{2}\leq\theta<\pi.
\end{cases}
\end{equation}
Since $\phi$ is convex in $(0,\pi)$, then $\phi'$ is increasing in $(0,\pi)$, and recalling that $\alpha\in(0,\frac{\pi}{2})$ we obtain that
\begin{equation*}
    \phi'(\theta)
    \leq
    \phi'(\pi-\alpha)
    \leq
    \frac{\pi}{(\sin\alpha)^{n-2}}
\end{equation*}
for every $0\leq\theta\leq\pi-\alpha$, which is the second inequality in \eqref{bounds-phi}.

On the other hand, for $\frac{\pi}{2}\leq\theta<\pi$ we get
\begin{equation*}
    \int_{\frac{\pi}{2}}^\theta\frac{t}{(\sin t)^{n-2}}\ dt
    \leq
    \Big(\theta-\frac{\pi}{2}\Big)\max_{\frac{\pi}{2}\leq t\leq\theta}\bigg\{\frac{1}{(\sin t)^{n-2}}\bigg\}
    =
    \frac{\theta-\frac{\pi}{2}}{(\sin\theta)^{n-2}},
\end{equation*}
Integrating \eqref{phi-d1} we obtain
\begin{equation*}
    0
    \leq
    \phi(\theta)
    \leq
        \displaystyle \frac{\pi^2}{8}+\frac{\theta-\frac{\pi}{2}}{(\sin\theta)^{n-2}}
\end{equation*}
for every $\frac{\pi}{2}\leq\theta<\pi$. In particular, since $\phi$ is increasing,
\begin{equation*}
    \phi(\theta)
    \leq
    \phi(\pi-\alpha)
    \leq
    \frac{\pi^2}{8}+\frac{\frac{\pi}{2}-\alpha}{(\sin(\pi-\alpha))^{n-2}}
    \leq
    \frac{\pi^2}{8}+\frac{\pi}{2(\sin\alpha)^{n-2}},
\end{equation*}
and the first in equality \eqref{bounds-phi} follows.
\end{proof}

\begin{lemma}\label{mainthm}
For $\alpha\in(0,\frac{\pi}{2})$ let
\begin{equation*}
    \Omega_\alpha
    :\,=
    \set{x\in\R^n}{x_1>-|x|\cos\alpha}
\end{equation*}
and $U:\overline\Omega_\alpha\to\R$ be the function defined as
\begin{equation}\label{maindef}
    \begin{cases}
        U(x)
        =
        |x|^\gamma\big(A-\phi(\theta)\big),
        \\
        \theta=\arccos\big(\frac{x_1}{|x|}\big),
    \end{cases}
\end{equation}
where $\phi:(-\pi,\pi)\to[0,\infty)$ is the auxiliary function
defined in \eqref{defphi} and $A>0$, $\gamma\in(0,\frac{1}{2}]$ are
constants satisfying
\begin{equation}\label{Ak}
    \frac{\pi^2}{8} + \frac{3\pi^2+\pi}{2(\sin\alpha )^{n-2}}
    \leq
    A
    \leq
    \frac{1}{\gamma(\gamma+n-2)}.
\end{equation}
Then $U\in C^2(\Omega_\alpha)\cap C(\overline\Omega_\alpha)$, $U(0)
= 0$, $U >0$ in $\overline\Omega_{\alpha} \setminus \{ 0 \}$,
\begin{equation}\label{U-supersolution}
    U(x)
    \geq
    \dashint_{B(x,\varrho)}U(y)\ dy
    \quad , \qquad
    U(x)
    \geq
    \frac{1}{2}\sup_{B(x,\varrho)}U +\frac{1}{2}\inf_{B(x,\varrho)}U
\end{equation}
for every ball $B(x,\varrho)\subset\Omega_{\alpha}$ and
\begin{equation}\label{u-bounds}
    \alpha^{2-n}\,|x|^\gamma
    \leq
    U(x)
    \leq
    \gamma^{-2}|x|^\gamma
\end{equation}
for every $x\in\overline\Omega_\alpha$. In particular, for each
$p\in [2, \infty ]$ and any admissible radius function $\rho$ in
$\Omega_{\alpha}$, $U$ is, simultaneously, a  $\T_{\rho ,
p}$-barrier  and a barrier for the $p$-laplacian at $0$ in
$\Omega_{\alpha}$.
\end{lemma}

\subsection{Proof of \Cref{mainthm}}

The regularity of $U$ is a direct
consequence of its construction. To see \eqref{u-bounds} we recall
\eqref{bounds-phi} together with \eqref{Ak} to get that, for every
$0\leq\theta\leq\pi-\alpha$,
\begin{equation*}
    0
    <
    \frac{3\pi^2}{2(\sin\alpha)^{n-2}}
    \leq
    A-\phi(\theta)
    \leq
    \frac{1}{\gamma(\gamma+n-2)}.
\end{equation*}
Then \eqref{u-bounds} follows.
\newline

In order to show \eqref{U-supersolution}, let us recall from
\cite[Lemma 2.4]{HAY-KEN} the expression of the laplacian of $U$ in
the polar coordinates $x\leftrightarrow(R,\theta)$,
\begin{equation*}
    \Delta U
    =
    R^{\gamma-2}\big[-\phi''(\theta)-(n-2)\phi'(\theta)\cot\theta + \gamma(\gamma+n-2)\big(A-\phi(\theta)\big)\big],
\end{equation*}
which, together with \eqref{edo},  gives
\begin{equation}\label{Laplacian}
    \Delta U
    =
    -R^{\gamma-2}\big[1-\gamma(\gamma+n-2)\big(A-\phi(\theta)\big)].
\end{equation}
Since $\phi\geq 0$ and $A$ and $\gamma$ satisfy \eqref{Ak}, it turns
out that $\Delta U\leq 0$. That is, $U$ is superharmonic and the
first inequality in \eqref{U-supersolution} follows by the mean
value property for superharmonic functions.
\newline

Before proving the second inequality in \eqref{U-supersolution}, we
first note that, since $U$ is rotationally invariant with respect to
the $x_1$-axis, the problem is actually bidimensional. Therefore we
replace $x\in\R^n$ by the complex number $z=Re^{i\vartheta}$, where
$R = |x|$, $\cos \vartheta = \frac{x_1}{|x|}$ and assume that
$\Omega_\alpha$ lies in the complex plane, so
\begin{equation*}
    \Omega_\alpha
    =
    \set{z=Re^{i\vartheta}}{R>0,\,|\vartheta|<\pi-\alpha}.
\end{equation*}
Then the second inequality in \eqref{U-supersolution} is equivalent
to
\begin{equation}\label{sup-inf n=2}
    U(z_0)
    \geq
    \frac{1}{2}\sup_{B(z_0,r)}U+\frac{1}{2}\inf_{B(z_0,r)}U
\end{equation}
for each $z_0=R_0e^{i\vartheta_0}$ and $0<r<R_0$ such that $B(z_0,r)\subset\Omega_\alpha$. Here we assume, by symmetry, that $0\leq\vartheta_0<\pi-\alpha$.

Observe that $\overline{B}(z_0 ,r) $ lies in the cone $\set{Re^{i\vartheta}}{|\vartheta-\vartheta_0| \leq t_m}$, where
\begin{equation*}
    t_m
    =
    \arcsin\Big(\frac{r}{R_0}\Big).
\end{equation*}
Given $|t|\leq t_m$, elementary computations show that the ray
$\set{Re^{i(\vartheta_0+t)}}{R>0}$ intersects $\partial B(z_0,r)$
at two points
$R_+(t)e^{i(\vartheta_0+t)} $ and $R_-(t)e^{i(\vartheta_0+t)}$,
where
\begin{equation}\label{R+- n=2}
    R_\pm(t)
    =
    R_0\bigg(\cos t\pm\sqrt{\Big(\frac{r}{R_0}\Big)^2-\sin^2t}\,\bigg).
\end{equation}

By \Cref{phi}, $\phi$ is increasing and even, $\phi\geq 0$ and
$\phi(0)= 0$. It follows that $\displaystyle \sup_{B(z_0,r)}U$ must
be of the form $R_+^\gamma(t)(A -\phi(\vartheta_0-t)) $ for some $0
\leq t \leq t_m$. Then
\begin{equation*}
    \sup_{B(z_0,r)}U+\inf_{B(z_0 ,r)}U
    \leq
    R_+^\gamma(t)\big(A-\phi(\vartheta_0-t)\big)+R_-^\gamma(t)\big(A-\phi(\vartheta_0+t)\big),
\end{equation*}
and, since $U(z_0)=R_0^\gamma\big(A-\phi(\vartheta_0)\big)$ by
definition, the desired  inequality \eqref{sup-inf n=2} will follow
from the next lemma.

\begin{lemma}
Let $A>0$ and $\gamma\in(0,\frac{1}{2}]$ satisfy \eqref{Ak}. For
$z_0=R_0e^{i\vartheta_0}$ and $0<r<R_0$ such that
$B(z_0,r)\subset\Omega_\alpha$, the inequality
\begin{equation}\label{key-ineq}
    R_+(t)^\gamma\big(A-\phi(\vartheta_0-t)\big)+R_-(t)^\gamma\big(A-\phi(\vartheta_0+t)\big)
    \leq
    2R_0^\gamma\big(A-\phi(\vartheta_0)\big)
\end{equation}
holds for every $|t|\leq\arcsin\big(\frac{r}{R_0}\big)$, where $R_\pm(t)$ were defined in \eqref{R+- n=2}.
\end{lemma}

\begin{proof}
Let us denote
\begin{equation*}
    \lambda_\pm
    =
    \lambda_\pm(t)
    =
    \frac{1}{2}\pare{\frac{R_\pm(t)}{R_0}}^\gamma
\end{equation*}
for simplicity. Then \eqref{key-ineq} is equivalent to
\begin{equation*}
    F(t)
    :\,=
    \frac{\phi(\vartheta_0)-\big(\lambda_+\phi(\vartheta_0-t)+\lambda_-\phi(\vartheta_0+t)\big)}{1-\big(\lambda_++\lambda_-\big)}
    \leq
    A
\end{equation*}
for every $0\leq t\leq\arcsin\big(\frac{r}{R_0}\big)$. We show that
the previous inequality holds true. Observe that after a
rearrangement of the terms we can write
\begin{equation*}
    F(t)
    =
    \phi(\vartheta_0)+\frac{\lambda_++\lambda_-}{1-(\lambda_++\lambda_-)}\Big[\phi(\vartheta_0)-\frac{\lambda_+}{\lambda_++\lambda_-}\phi(\vartheta_0-t)-\frac{\lambda_-}{\lambda_++ \lambda_-}\phi(\vartheta_0+t)\Big].
\end{equation*}
Let us focus on the term in brackets. From the convexity of $\phi$ we can estimate the term in brackets as follows
\begin{multline*}
    \phi(\vartheta_0)-\frac{\lambda_+}{\lambda_++\lambda_-}\phi(\vartheta_0-t)-\frac{\lambda_-}{\lambda_++ \lambda_-}\phi(\vartheta_0+t)
    \\
    \leq
    \phi(\vartheta_0)-\phi\Big(\vartheta_0-\frac{\lambda_+-\lambda_-}{\lambda_++\lambda_-}\,t\Big)
    \leq
    \frac{\lambda_+-\lambda_-}{\lambda_++\lambda_-}\,t\phi'(\vartheta_0).
\end{multline*}
Thus
\begin{equation*}
    F(t)\leq
    \phi(\theta_0)+\frac{\lambda_+-\lambda_-}{1-(\lambda_++\lambda_-)}\,t\phi'(\theta_0).
\end{equation*}
Notice that, since the function $\phi$ is increasing in $(0,\pi)$
and $\vartheta_0\geq 0$ by assumption, then $\phi'(\vartheta_0)\geq
0$. Next, using \Cref{aux-lemma} (see \Cref{appendix}) we get
\begin{equation*}
    \lambda_+\pm\lambda_-
    =
    \frac{R_+(t)^\gamma+R_-(t)^\gamma}{2R_0^\gamma}
    \leq
    \frac{1}{2}\Big(1+\frac{r}{R_0}\Big)^\gamma\pm\frac{1}{2}\Big(1-\frac{r}{R_0}\Big)^\gamma,
\end{equation*}
which together with $t\leq\arcsin\big(\frac{r}{R_0})\leq\frac{\pi r}{2R_0}$ yields
\begin{equation*}
    F(t)\leq
    \phi(\vartheta_0)+\frac{\pi}{2}\cdot\frac{\displaystyle \frac{r}{2R_0}\Big[\Big(1+\frac{r}{R_0}\Big)^\gamma-\Big(1-\frac{r}{R_0}\Big)^\gamma\Big]}
    {\displaystyle 1-\frac{1}{2}\Big[\Big(1+\frac{r}{R_0}\Big)^\gamma+\Big(1-\frac{r}{R_0}\Big)^\gamma\Big]}\,\phi'(\vartheta_0).
\end{equation*}
By \Cref{<4} together with the fact that $\gamma\in(0,\frac{1}{2}]$
we get
\begin{equation*}
    F(t)\leq
    \phi(\vartheta_0)+2\pi\phi'(\vartheta_0)
    \leq
    \frac{\pi^2}{8}+\frac{3\pi^2+\pi}{2(\sin\alpha)^{n-2}}.
\end{equation*}
where in the second inequality we have recalled the estimates \eqref{bounds-phi}. Then the result follows from the choice of $A$ in \eqref{Ak}.
\end{proof}

\begin{remark}
We want to emphasize that, in the proof of \Cref{mainthm}, the
definition of $\phi$ as solution of the differential equation
\eqref{edo} is used exclusively to show the first inequality in
\eqref{U-supersolution}, while for the second inequality  we only
need to require the convexity of $\phi$ in $(-\pi,\pi)$ and the fact
that $\phi$ is increasing in $[0,\pi)$.
\end{remark}

The following proposition says that the function $U$ is also a
$p$-superharmonic for each $p\in [2, \infty]$.
\begin{proposition}
Let $U$ be the function defined in \eqref{maindef} with $A>0$ and
$\gamma\in(0,\frac{1}{2}]$ as in \eqref{Ak}. Then $\Delta_p U\leq 0$
in $\Omega_\alpha$ for each $p\in [2,\infty]$.
\end{proposition}

\begin{proof}
From the representation \eqref{p-lapla} and the fact that $ p \geq
2$ it is enough to check that $\Delta U\leq 0$ and $\Delta_\infty U \leq 0$.

The choices of $A$ and $\gamma$ in the expression of $\Delta U$  in \eqref{Laplacian} easily give that $\Delta U \leq 0$.  We also need the expression of $\Delta_\infty U$ in polar coordinates (see \cite{DEB-SMI}):
\begin{equation*}
    \Delta_\infty U
    =
    -R^{3\gamma-4}\Big[\gamma^3(1-\gamma)\big(A-\phi\big)+\gamma(1-2\gamma)\big(A-\phi\big)(\phi')^2+(\phi')^2\phi''\Big].
\end{equation*}
Observe that, since $\gamma\in(0,\frac{1}{2}]$, $A-\phi>0$ and $\phi''\geq 0$, the term in brackets is positive, so $\Delta_\infty U\leq 0$.
\end{proof}

\subsection{Barrier for the complement of a truncated cone}\label{Truncated Cone}

Let $\alpha\in(0,\frac{\pi}{2})$, $r>0$ and define
\begin{equation*}
    \Omega_{\alpha,r} = \R^n \setminus K_{\alpha ,r}
\end{equation*}
where $K_{\alpha, r}$ is as in \eqref{Kr}. Note that $\Omega_{\alpha
,r}$  is the complement of a truncated cone and that
$\Omega_{\alpha,r}\supset\Omega_\alpha$. Let $U:\Omega_\alpha\to\R$
be the function defined in \eqref{maindef} for $A>0$ and
$\gamma\in(0,\frac{1}{2}]$ as in \eqref{Ak}. Then, from the first
inequality in \eqref{u-bounds} it follows that
\begin{equation*}
    m
    =
    \inf\set{U(x)}{x\in \Omega_\alpha\setminus B(0,r)}
    \geq
    \alpha^{2-n}r^\gamma
    >
    0.
\end{equation*}

\begin{lemma}\label{truncated}
Let $U$ be as in \eqref{maindef} and define $w:\Omega_{\alpha,r}\to\R$ by
\begin{equation}\label{w}
    w(x)
    =
    \begin{cases}
        \min\{U(x),m\}
        & \text{ if } x\in\Omega_{\alpha,r}\cap B(0,r),
        \\
        m
        & \text{ if } x\in\Omega_{\alpha,r}\setminus B(0,r).
\end{cases}
\end{equation}

Then $w\in C(\overline\Omega_{\alpha,r})$, $w(0)=0$, $w>0$ in
$\overline{\Omega_{\alpha ,r}} \setminus \{ 0 \}$,
\begin{equation}\label{w-supersolution}
    w(x)
    \geq
    \dashint_{B(x,\varrho)}w(y)\ dy
    \quad , \qquad
    w(x)
    \geq
    \frac{1}{2}\sup_{B(x,\varrho)}w+\frac{1}{2}\inf_{B(x,\varrho)}w
\end{equation}
for every ball $B(x,\varrho)\subset\Omega_{\alpha,r}$ and
\begin{equation}\label{w-bounds}
    \L(|x|)
    \leq
    w(x)
    \leq
    \gamma^{-2}|x|^\gamma
\end{equation}
for every $x\in\overline\Omega_{\alpha,r}$, where
\begin{equation}\label{def L}
    \L(t)
    :\,=
    \alpha^{2-n}\,\min\{t,r\}^\gamma.
\end{equation}
In particular, for each $p\in [2, \infty]$ and any admissible radius
function in $\Omega_{\alpha, r}$, $w$ is, simultaneously, a
$\T_{\rho, p}$-barrier and a barrier for the $p$-laplacian at $0$ in
$\Omega_{\alpha,r}$.
\end{lemma}

\begin{proof}
Since $U\in C(\overline\Omega_\alpha)$, the continuity of $w$ only
needs to be checked at $\Omega_\alpha\cap\partial B(0,r)$. Fix
$x_0\in\Omega_\alpha\cap\partial B(0,r)$. Then $U(x_0)\geq m$. From
the continuity of $U$ it follows that $\displaystyle \lim_{x\to
x_0}w(x)=\min \{U(x_0 ),m\}=m=w(x_0)$. The inequalities in
\eqref{w-bounds} follow from the definition of $w$ and
\eqref{u-bounds}. To prove  \eqref{w-supersolution}, choose any ball
$B(x,\rho)\subset\Omega_{\alpha,r}$. We distinguish two cases:
\begin{enumerate}[label=\arabic*)]
\item If $w(x)<m$ then observe that $B(x,\varrho)\subset\Omega_\alpha$
and that $w\leq U$ in $B(x,\varrho)$. It follows from \Cref{mainthm}
that
\begin{equation*}
    \dashint_{B(x,\varrho)}w(y)\ dy
    \leq
    \dashint_{B(x,\varrho)}U(y)\ dy
    \leq
    U(x)
    =
    w(x)
\end{equation*}
and
\begin{equation*}
    \frac{1}{2}\sup_{B(x,\varrho)}w+\frac{1}{2}\inf_{B(x,\varrho)}w
    \leq
    \frac{1}{2}\sup_{B(x,\varrho)}U+\frac{1}{2}\inf_{B(x,\varrho)}U
    \leq
    U(x)
    =
    w(x).
\end{equation*}
\item If $w(x)=m$ then \eqref{w-supersolution}
follows immediately since $w\leq m$.
\end{enumerate}

\

From \eqref{w-supersolution} it is immediate that $\T_{\rho, p}w\leq
w$ for every admissible radius function $\rho$ in
$\Omega_{\alpha,r}$ and each $p\in [2, \infty]$. Finally, that $w$
is $p$-superharmonic is consequence of the $p$-superharmonicity of
$U$, the invariance of $p$-superharmonic functions by rotations and
the Pasting Lemma (\cite[Lemma 7.9]{HKM}).

\end{proof}

\subsection{Proof of \Cref{mainthmloc}}\label{ECC}
Let $\Omega \subset \R^n$ be a  bounded domain satisfying the
uniform exterior cone condition with constants $\alpha \in (0,
\frac{\pi}{2})$ and $r>0$. From \eqref{gamma} it follows in
particular that  $\gamma \in (0, \frac{1}{2}]$ and
\begin{equation*}
    \frac{1}{\gamma(\gamma+n-2)}
    >
    \frac{1}{\gamma n}
    >
    \frac{13\pi^2+4\pi}{8(\sin\alpha)^{n-2}}
    \geq
    \frac{\pi^2}{8}+\frac{3\pi^2+\pi}{2(\sin\alpha)^{n-2}},
\end{equation*}
which allows to choose $A>0$ so that \eqref{Ak} holds and,
subsequently, to construct $U$ and $w$ as in \eqref{maindef} and
\eqref{w}, respectively.

Recalling \Cref{DEF-ECC}, there exists, for every
$\xi\in\partial\Omega$, a rotation $R_\xi\in SO(n)$ in $\R^n$ such
that $\xi+R_\xi(K_{\alpha,r})\subset\R^n\setminus\Omega$ or,
equivalently,
\begin{equation*}
    R_\xi^\top(\Omega-\xi)
    \subset
    \Omega_{\alpha,r}.
\end{equation*}
Then we define $w_\xi:\overline\Omega\to\R$ by
\begin{equation}\label{Trho-barrier}
    w_\xi(x)
    =
    w\big(R_\xi^\top(x-\xi)\big),
\end{equation}
where $w$ is the barrier function given by \eqref{w}. Recalling
\Cref{truncated}, we observe that
$w_\xi$ is non-negative in $\overline\Omega$ and $w_\xi(x)=0$ if and
only if $x=\xi$. On the other hand, if $\rho$ is an admissible
radius function in $\Omega$, since
$\xi+R_\xi\big(B_\rho(x)\big)=B(\xi+R_\xi(x),\rho(x))\subset\Omega_{\alpha,r}$,
then
\begin{equation*}
    \T_\rho w_\xi(x)
    =
    \T_\rho w\pare{R_\xi^\top(x-\xi)}
    \leq
    w\pare{R_\xi^\top(x-\xi)}
    =
    w_\xi(x),
\end{equation*}
Thus \eqref{MS-supersolution} follows from \eqref{w-supersolution}
and, in particular, $w_\xi$ is a $\T_\rho$-barrier at
$\xi\in\partial\Omega$. The fact that $w_{\xi}$ is also a barrier
for the $p$-laplacian follows in a similar way. Finally, from
\eqref{w-bounds} we get \eqref{lower-upper-control} for every $x\in\overline\Omega$, where $\L$ is given by \eqref{def L}. This finishes the proof of the theorem.

\section{Existence of solutions}\label{sec.existence}

We split the section in two parts. In the first part we show the equicontinuity of the sequence $\{\T_\rho^k u\}$ in $\overline\Omega$. In the second one we establish existence and uniqueness of the Dirichlet Problem for $\T_\rho$.

\subsection{Equicontinuity results}

At this point we refer to \cite[Theorem 4.5]{ARR-LLO-18} for the equicontinuity of the sequence $\{\T_\rho^k\}_k$ at
interior points of $\Omega$, where $u\in \mathcal{K}_f$ (see also \cite[Proposition 2.6]{ARR-LLO-16-1}).

\begin{theorem}[{\cite[Theorem 4.5]{ARR-LLO-18}}]\label{THM-LOCAL-EQUICONTINUITY}
Let $\Omega\subset\R^n$ be a bounded domain and $p\in[2,\infty)$. Suppose that $\rho\in C(\overline\Omega)$ is a
continuous admissible radius function in $\Omega$ satisfying
\begin{equation*}
    \lambda\dist(x,\partial\Omega)^\beta
    \leq
    \rho(x)
    \leq
    \Lambda\dist(x,\partial\Omega)
\end{equation*}
for all $x\in\Omega$, where
\begin{equation*}
    \beta
    \geq
    1,
    \qquad
    0
    <
    \Lambda
    <
    1-\pare{\frac{p-2}{n+p}}^{1/\beta}
    \qquad\text{ and }\qquad
    0
    <
    \lambda
    \leq
    (\diam\Omega)^{1-\beta}\Lambda.
\end{equation*}
Then, for any $u\in C(\overline\Omega)$, the sequence of iterates $\{\T_\rho^k u\}_k$ is locally uniformly equicontinuous
in $\Omega$.
\end{theorem}

Therefore it only remains to show that, given a function $u\in \mathcal{K}_f$, the sequence of iterates $\{\T_\rho^k u\}_k$
is equicontinuous at each $\T_\rho$-regular point of the boundary.

\begin{proposition}\label{THM-BOUNDARY-EQUICONTINUITY}
Let $\Omega\subset\R^n$ be a bounded domain and $f\in C(\partial\Omega)$. For any $u\in \mathcal{K}_f$, the sequence of
iterates $\{\T_\rho^k u\}_k$ is equicontinuous at each $\T_\rho$-regular point of $\partial\Omega$.
\end{proposition}

\begin{proof}
Since $u\in \mathcal{K}_f$, then $u$ is uniformly continuous in $\overline\Omega$, that is, for each $\eta>0$ there exists
small enough $\delta>0$ such that $|u(x)-u(y)|<\eta$ for every $x,y\in\overline\Omega$ satisfying $|x-y|<\delta$.
Fix $C=C_{u,\eta}=2\|u\|_\infty/\L(\delta)$, where $\L$ is the non-decreasing continuous function defined in \eqref{def L}.
Then
\begin{equation*}
    |u(x)-u(y)|
    \leq
    C\L(|x-y|)+\eta
\end{equation*}
for every $x,y\in\overline\Omega$. Therefore, if
$\xi\in\partial\Omega$ is  $\T_\rho$-regular we obtain, recalling
\eqref{lower-upper-control}, that
\begin{equation*}
                |u(x)-f(\xi)|
                \leq
                Cw_\xi(x)+\eta
\end{equation*}
 for every $x\in\overline\Omega$, where $w_\xi$ is a $\T_\rho$-barrier at $\xi$.
 Let $k\in\N$. By the affine invariance and the monotonicity of $\T_\rho^k$,

\begin{equation*}
    \big|\T_\rho^k u(x)-f(\xi)\big|
    =
    \big|\T_\rho^k\big(u-f(\xi)\big)(x)\big|
    \leq
    \T_\rho^k\big(Cw_\xi+\eta\big)(x)
    \leq
    Cw_\xi(x)+\eta
\end{equation*}
for every $x\in\overline\Omega$, where in the second inequality we have used that $w_\xi\geq\T_\rho w_\xi$.
Finally, taking limits it turns out that
\begin{equation*}
                0
                \leq
                \limsup_{x\to\xi}\big|\T_\rho^ku(x)-f(\xi)\big|
                \leq
                C\limsup_{x\to\xi}w_\xi(x)+\eta
                =
                \eta
\end{equation*}
for each $k\in\N$ and every $\eta>0$. Thus the sequence of iterates $\{\T_\rho^ku\}_k$ is equicontinuous at $\xi$
and the proof is finished.
\end{proof}

\begin{remark}
The proof of the above result only requires the affine invariance
and the monotonicity of $\T_\rho$. On the other hand, the proof does
not require any assumption on the admissible radius function. Thus,
the equicontinuity estimates obtained in the previous result are
independent  of the particular choice of $\rho$ in the definition of
the operator $\T_\rho$.
\end{remark}

In view of \Cref{THM-LOCAL-EQUICONTINUITY} and \Cref{THM-BOUNDARY-EQUICONTINUITY}, we have proved the following.

\begin{theorem}\label{THM-GLOBAL-EQUICONTINUITY}
Under the assumptions in \Cref{THM-LOCAL-EQUICONTINUITY}, assume in addition that $\Omega$ is $\T_\rho$-regular. If  $u\in \mathcal{K}_f$, then the sequence of iterates $\{\T_\rho^k u\}_k$ is equicontinuous in $\overline\Omega$.
\end{theorem}

\subsection{Existence and uniqueness}

We start with the following comparison principle, that uses a standard argument (see also \cite[Proposition 4.1]{ARR-LLO-16-1}).

\begin{proposition}\label{THM-COMPARISON-PPLE}
Let $\Omega\subset\R^n$ be a bounded domain and $\rho$ an admissible radius function in $\Omega$. Assume that
$u,v\in C(\overline\Omega)$ satisfy $u\leq\T_\rho u$, $v\geq\T_\rho v$ in $\Omega$ and $u\leq v$ on $\partial\Omega$.
Then $u\leq v$ in $\Omega$.
\end{proposition}

\begin{proof}
Let $m=\max_{\overline\Omega}(u - v)$. We show that $m\leq 0$ by contradiction: suppose that $m>0$ and let
$A:\,=\set{x\in \Omega}{u(x)-v(x)=m}$. Since $u-v$ is upper semicontinuous in $\overline\Omega$ and $u-v\leq 0$
on the boundary, then $A$ is a nonempty closed subset of $\Omega$. The contradiction will then follow by proving that $A$
is also open, so $A=\Omega$ and $u(x)-v(x)=m>0$ for every $x\in\Omega$.

To see that $A$ is open, we choose any $a\in A$ and we show that
$B_\rho(a)\subset A$. Recalling that $u$ and $v$ are sub and
super-solutions of $\T_\rho$ we obtain that
\begin{equation*}
                \T_\rho u(a)
                \geq
                u(a)
                =
                m+v(a)
                \geq
                m+\T_\rho v(a),
\end{equation*}
and by the definition of $\T_\rho$,
\begin{equation*}
                \frac{p-2}{n+p}\,\S_\rho u(a)+\frac{n+2}{n+p}\,\M_\rho u(a)
                \geq
                \frac{p-2}{n+p}\,\big(m+\S_\rho v(a)\big)+\frac{n+2}{n+p}\,\big(m+\M_\rho v(a)\big).
\end{equation*}
Hence, by the monotonicity of $\S_\rho$ and $\M_\rho$, and since $p\in[2,\infty)$, it turns out that
\begin{equation*}
                \M_\rho u(a)
                =
                m+\M_\rho v(a),
\end{equation*}
and recalling the definition of $\M_\rho$,
\begin{equation*}
                m
                =
                \dashint_{B_\rho(a)} (u-v).
\end{equation*}
Since $m$ is defined as the maximum in $\overline\Omega$ of $u-v$, then $u(x)-v(x)=m$ for every $x\in B_\rho(a)$.
Then $B_\rho(a)\subset A$, and so $A$ is an open set. Therefore, since $\Omega$ is connected, $A=\Omega$ and $u-v\equiv m>0$
in $\Omega$, which contradicts the assumption $u\leq v$ on $\partial \Omega$.
\end{proof}

Uniqueness of fixed points follows immediately as a corollary.

\begin{corollary}\label{THM-UNIQUENES}
Let $\Omega\subset\R^n$ be a bounded domain and  $f\in C(\partial\Omega)$.  Suppose that
$u,v\in \mathcal{K}_f$ are fixed points of $\T_\rho$. Then $u=v$ in
$\overline\Omega$.
\end{corollary}

In order to show the existence of fixed points for $\T_\rho$ in $\mathcal{K}_f$, we will make use of the following
technical result, which can be stated in the more general context of Banach spaces.

\begin{lemma}\label{Banach}
Let $(X,\norm{\cdot})$ be a Banach space, $\emptyset\neq K\subset X$ any closed subset and $T:K\rightarrow K$ a
non-expansive operator. Fix $x\in K$. If $y\in K$ is any limit point of the sequence $\{T^k x\}_k$ then
\begin{equation}\label{asymptotic-regular-C}
    \lim_{k\rightarrow\infty}\|T^{k+1}x-T^k x\|
    =
    \|Ty-y\|.
\end{equation}
\end{lemma}

\begin{proof}
Observe first that, since $T$ is non-expansive, the sequence of non-negative real numbers $\{\|T^{k+1}x-T^k x\|\}_k$ is
non-increasing, and thus every subsequence converges to the same limit. Next, take any convergent subsequence
$\displaystyle \{T^{k_j} x\}_j$ and denote the limit by $y\in K$.  The triangle inequality and the non-expansiveness of $T$ yield
\begin{equation*}
\begin{split}
\Big|\|T^{k_j+1}x-T^{k_j}x\|-\|Ty-y\|\Big|
                \leq
                ~&
\|(T^{k_j+1}x-Ty)-(T^{k_j}x-y)\|
                \\
                \leq
                ~&
                2\|T^{k_j}x-y\|,
\end{split}
\end{equation*}
for each $j\in\N$. Then \eqref{asymptotic-regular-C} follows after taking limits as $j\to\infty$.
\end{proof}

Now we are ready to prove \Cref{MAIN-THM-EXISTENCE}.

\begin{proof}[Proof of \Cref{MAIN-THM-EXISTENCE}]
Since $\Omega$ is $\T_\rho$-regular by assumption,  the sequence of
iterates $\{\T_\rho^k u\}_k$ is equicontinuous at each point in
$\overline\Omega$ for any $u\in \mathcal{K}_f$, by
\Cref{THM-GLOBAL-EQUICONTINUITY}. Then the Arzel\`{a}-Ascoli's
theorem yields the existence of at least one subsequence converging
uniformly to a function $v\in \mathcal{K}_f$. Furthermore,
$\T_\rho^\ell v$ is also a limit point of $\{\T_\rho^k u\}_k$ for
each $\ell=0,1,2,\ldots$ Therefore, since $\mathcal{K}_f$ is a
closed subset of $C(\overline\Omega)$ and
$\T=\T_\rho:\mathcal{K}_f\rightarrow \mathcal{K}_f$ is
non-expansive, \Cref{Banach} implies that
\begin{equation}\label{norm-C}
                \|\T^{\ell+1} v-\T^\ell v\|_\infty
                =
                \lim_{k\rightarrow\infty}\|\T^{k+1}u-\T^k u\|_\infty
                =\,:
                d
                \geq 0
\end{equation}
for every $\ell=0,1,2,\ldots$ In consequence, if $d=0$ we have in particular that $\T_\rho v=v$, so $v$ is a fixed point
of $\T_\rho$.

In order to show that actually $d=0$, let us assume on the contrary that $d>0$ and argue by contradiction. Let $\ell\in\N$
to be fixed later. Since $\T^{\ell+1} v-\T^\ell v$ is a continuous function vanishing on $\partial\Omega$, we can choose an
interior point $x_0\in\Omega$ such that
\begin{equation*}
                |\T^{\ell+1} v(x_0)-\T^\ell v(x_0)|
                =
                d.
\end{equation*}
We assume that $\T^{\ell+1} v(x_0)-\T^\ell v(x_0)=d$ since otherwise the proof goes in an analogous way. Recalling the
definition of $\T=\T_\rho$ and $\M=\M_\rho$, it turns out that
\begin{equation}\label{C-opT}
    d
    =
    \frac{p-2}{n+p}\brackets{\S(\T^\ell v)(x_0)-\S(\T^{\ell-1} v)(x_0)}+\frac{n+2}{n+p}\,\M\big(\T^\ell v-\T^{\ell-1} v\big)(x_0).
\end{equation}

From \eqref{norm-C} and the non-expansiveness of $\S$ and $\M$, it
follows that
\begin{equation*}
    \S(\T^\ell v)(x_0)-\S(\T^{\ell-1} v)(x_0)
    \leq
    d
    \quad \text{ and } \quad
    \M\big(\T^\ell v-\T^{\ell-1} v\big)(x_0)
    \leq
    d,
\end{equation*}
which, together with \eqref{C-opT}, implies that $\M(\T^\ell v-\T^{\ell-1} v)(x_0)=d$.

Equivalently,
\begin{equation*}
                \dashint_{B_\rho(x_0)}\big(\T^\ell v(y)-\T^{\ell-1}v(y)\big)\ dy
                =
                d,
\end{equation*}
and by \eqref{norm-C}, the integrand
 must be equal to $d$ in $B_\rho(x_0)$. In particular, $\T^\ell v(x_0)-\T^{\ell-1} v(x_0) = d $ so we can repeat
 this argument iteratively until we finally get
that
\begin{equation*}
                \T^\ell v(x_0)
                =
                v(x_0)+\ell d.
\end{equation*}
Recalling \eqref{Kf} and  that
$\T=\T_\rho:\mathcal{K}_f\rightarrow \mathcal{K}_f$
 we get
\begin{equation*}
                \|f\|_\infty
                \geq
                \T^\ell v(x_0)
                =
                v(x_0)+\ell d
                \geq
                -\|f\|_\infty+\ell d.
\end{equation*}
Hence, choosing $\ell$ such that
\begin{equation*}
                \ell
                >
                \frac{2\|f\|_\infty}{d}
\end{equation*}
we obtain the desired contradiction.

Finally, to see that the sequence of iterates $\{\T_\rho^k u\}_k$
actually converges uniformly to the unique fixed point $v\in
\mathcal{K}_f$, suppose on the contrary that there exist $\eta>0$ and a subsequence
$\{\T_\rho^{k_j}u\}_j$ such that $
                \|\T_\rho^{k_j}u-v\|_\infty
                \geq
                \eta
$ for each $j\in\N$. We can assume that this subsequence converges
uniformly to a function $w\in \mathcal{K}_f$ (otherwise, by
equicontinuity and Arzel\`{a}-Ascoli's theorem we could take a
further subsequence), which would
be a limit point of $\{\T_\rho^ku\}_k$ in $\mathcal{K}_f$, and thus
a fixed point of $\T_\rho$. Then the contradiction follows by
uniqueness and the fact that $\|w-v\|_\infty\geq\eta$.
\end{proof}

\section{Convergence to $p$-harmonic functions} \label{sec.convergence}

In this section we study the convergence of solutions $u_\rho$ to \eqref{DP} as the admissible radius function converges to zero in $\Omega$. Before moving into details,
it is worth to recall that one of the main connections between mean
value properties and $p$-harmonic functions arises from the
asymptotic expansion for the $p$-laplacian of a twice-differentiable
function $\phi$ at a non-critical point $x$. This expansion can be
expressed in terms of the average operator $\T_\rho$ as follows
\begin{equation*}
                \T_\rho\phi(x)
                =
                \phi(x)
                +\frac{\rho(x)^2}{2(n+p)}\,\Delta^N_p\phi(x)+o(\rho(x)^2)
                \qquad
                (\rho(x)\to 0),
\end{equation*}
where $\Delta_p^N\phi$ stands for the \emph{normalized $p$-laplacian} of $\phi$ defined as
\begin{equation*}
    \Delta_p^N\phi
    :\,=
    \Delta\phi + (p-2)\frac{\Delta_{\infty}\phi}{|\nabla \phi |^2}.
\end{equation*}

Heuristically speaking, if the fixed points $\T_\rho u_\rho=u_\rho$
converge to a function $u_0$ as $\rho\to 0$, then it is reasonable
to expect that this function is $p$-harmonic. Indeed, this is one of
the key ideas required in the proof of \Cref{MAIN-THM-CONVERGENCE}.

To this end, we need first to impose appropriate conditions in order
to ensure that $\rho (x)$
converges to zero in a uniform way. Given a bounded domain
$\Omega\subset\R^n$, let us consider a collection of continuous
admissible radius functions $\{\rho_\varepsilon\}_{0<\varepsilon\leq
1}$ satisfying
\begin{equation}\label{lambdas-epsilon}
                \lambda\dist(x,\partial\Omega)^\beta
                \leq
                \frac{\rho_\varepsilon(x)}{\varepsilon}
                \leq
                \Lambda\dist(x,\partial\Omega)
\end{equation}
for all $x\in\Omega$ and every $0<\varepsilon\leq 1$, where $\beta$,
$\lambda$ and $\Lambda$ are as in \eqref{ctts}. Since $\Omega$ is
bounded,  $\|\rho_\varepsilon\|_\infty$ decreases as fast as, at
least, a constant multiple of $\varepsilon$. In consequence
$\rho_\varepsilon(x)=O(\varepsilon )$ uniformly for every $x\in\Omega$, and the asymptotic expansion for
 $\T_{\rho_\varepsilon}$ becomes
\begin{equation}\label{taylor}
                \T_{\rho_\varepsilon}\phi(x)
                =
                \phi(x)
                +\frac{\varepsilon^2}{2(n+p)}\bigg(\frac{\rho_\varepsilon(x)}{\varepsilon}\bigg)^2\Delta^N_p\phi(x)+o(\varepsilon^2)
                \qquad
                (\varepsilon\to 0)
\end{equation}
for every $x\in\Omega$.

On the other hand, $\rho_\varepsilon$ is an admissible radius
function satisfying the hypothesis of \Cref{MAIN-THM-EXISTENCE} for
each $0<\varepsilon\leq 1$. Therefore, assuming that $\Omega$
satisfies the uniform exterior cone condition,
\Cref{MAIN-THM-EXISTENCE} yields, for any fixed $f\in
C(\partial\Omega)$, a function $u_\varepsilon\in C(\overline\Omega)$
satisfying
\begin{equation}\label{rhoeps}
                \begin{cases}
                \T_{\varepsilon} u_\varepsilon=u_\varepsilon & \mbox{ in } \Omega,
                \\
                u_\varepsilon=f & \mbox{ on } \partial\Omega,
                \end{cases}
\end{equation}
for each $0<\varepsilon\leq 1$, where $\T_{\varepsilon} :\,=
\T_{\rho_{\varepsilon}}$.
\\

The strategy to prove \Cref{MAIN-THM-CONVERGENCE} is inspired by the
method for convergence of numerical
schemes established by Barles and Souganidis in the 90's
(\cite{BAR-SOU}) and in a more recent result by del Teso, Manfredi
and Parviainen on the convergence of dynamic programming principles
for the $p$-laplacian (\cite{DEL-MAN-PAR}). The steps in the proof
can be split into two parts. First, by taking pointwise limits as
$\varepsilon\to 0$, we define semicontinuous functions
\begin{equation}\label{pointwise-limits}
    \underline u(x)
    :\,=
    \liminf_{y\to x,  \ \varepsilon\to 0}u_\varepsilon(y)
    \leq
    \limsup_{y\to x,\ \varepsilon\to 0}u_\varepsilon(y)
    =\,:
    \overline u(x),
\end{equation}
and, using the asymptotic expansion \eqref{taylor}, we show that
$\underline u$ and $\overline u$ are $p$-superharmonic and
$p$-subharmonic, respectively. In a second part we prove that
$\overline u\leq \underline u$ in $\overline\Omega$ with the aid of
the comparison principle for $p$-subharmonic and $p$-superharmonic
functions \cite[Theorem 2.7]{JUU-LIN-MAN}, so both functions
coincide with $u_0$, the unique $p$-harmonic function satisfying
\begin{equation*}
                \begin{cases}
                \Delta_pu_0=0
                & \text{ in } \Omega,
                \\
                u_0=f
                & \text{ on } \partial\Omega.
                \end{cases}
\end{equation*}
\\

We remind that, by \cite[Theorem 2.5]{JUU-LIN-MAN}, the concepts of
$p$-subharmonic function and viscosity $p$-subsolution coincide. Hence, for the sake of
simplicity, we will use hereafter the viscosity characterization.

\begin{definition}
Let $ p\in [2, \infty]$.
\begin{enumerate}
\item We say that an upper semicontinuous function $u$ in $\Omega$
is \emph{$p$-subharmonic} if $u\not\equiv-\infty$ and for every $x\in\Omega$ and any $\phi\in C^2(\Omega)$ such that $\nabla\phi(x)\neq 0$ and $u-\phi<u(x)-\phi(x)=0$ in $\Omega\setminus\{x\}$ we have that $\Delta_p\phi(x)\geq 0$.
\item We say that a lower semicontinuous function $u$ in $\Omega$
is \emph{$p$-superharmonic} if $u\not\equiv\infty$ and for every $x\in\Omega$ and any $\phi\in C^2(\Omega)$ such that $\nabla\phi(x)\neq 0$ and $u-\phi>u(x)-\phi(x)=0$ in $\Omega\setminus\{x\}$ we have that $\Delta_p\phi(x)\leq 0$.
\item $u\in C(\Omega)$ is \emph{$p$-harmonic} if it is both $p$-subharmonic and $p$-superharmonic.
\end{enumerate}
\end{definition}

\begin{proposition}\label{THM-p-SUBHARMONIC}
Let $p\in [2, \infty )$, $\Omega\subset\R^n$ a bounded domain
satisfying the uniform exterior cone condition, $f\in
C(\partial\Omega)$ and $u_\varepsilon\in \mathcal{K}_f$  the unique
solution of \eqref{rhoeps} provided by \Cref{MAIN-THM-EXISTENCE},
for  $0<\varepsilon\leq 1$. Let $\underline u$ and $\overline u$ be
the functions defined in \eqref{pointwise-limits}. Then $\underline
u$ is $p$-superharmonic and $\overline u$ is $p$-subharmonic in
$\Omega$.
\end{proposition}

\begin{proof}
We show that $\overline u$ is $p$-subharmonic. Fix any $x\in\Omega$ and $\phi\in C^2(\Omega)$ such that
$\nabla\phi(x)\neq 0$ and $\overline u-\phi<\overline u(x)-\phi(x)=0$ in $\Omega\setminus\{x\}$. That is, $x$ is a strict
global maximum of $\overline u-\phi$ in $\overline\Omega$. We need to check that $\Delta_p\phi(x)\geq 0$.

From the definition of $\overline u$, we pick sequences
$\varepsilon_j\to 0$ and $z_j\to x$ such that
$u_{\varepsilon_j}(z_j)\to\overline u(x)$. For each $j$, let
$y_j\in\overline\Omega$ such that $y_j$ is a maximum of
$u_{\varepsilon_j}-\phi$ with $\nabla\phi(y_j)\neq 0$. We claim that
$y_j\to x$ as $j\to\infty$. Otherwise, there would be a further
subsequence (still denoted by $\{y_j\}_j$) converging to $x'\neq x$.
Then
\begin{equation*}
    \overline u(x')-\phi(x')
    \geq
    \limsup_{j\to\infty}\big(u_{\varepsilon_j}(y_j)-\phi(y_j)\big)
    \geq
    \limsup_{j\to\infty}\big(u_{\varepsilon_j}(z_j)-\phi(z_j)\big)
    =
    \overline u(x)-\phi(x),
\end{equation*}
so we obtain a contradiction with the fact that $x$ is a strict global maximum of $\overline u-\phi$. Then $y_j\to x$ and $u_{\varepsilon_j}-u_{\varepsilon_j}(y_j)\leq\phi-\phi(y_j)$ in $\overline\Omega$. In consequence, by the monotonicity and the affine invariance of $\T_{\varepsilon_j}$ we obtain
\begin{equation*}
                0
                =
                \T_{\varepsilon_j}u_{\varepsilon_j}(y_j)-u_{\varepsilon_j}(y_j)
                \leq
                \T_{\varepsilon_j}\phi(y_j)-\phi(y_j)
\end{equation*}
for every $j\in\N$. Notice that the left-hand side of the inequality is equal to zero due to the fact that $\T_\varepsilon u_\varepsilon=u_\varepsilon$ for every $0<\varepsilon\leq 1$. Recall \eqref{taylor}, rearrange terms and divide by $\varepsilon_j^2$ to obtain
\begin{equation*}
                \bigg(\frac{\rho_{\varepsilon_j}(y_j)}{\varepsilon_j}\bigg)^2\Delta^N_p\phi(y_j)
                \geq
                o(1)
                \quad (j\to\infty).
\end{equation*}
Assume for a moment that $\Delta^N_p\phi(x)<0$. Then $\Delta^N_p\phi(y_j)<0$ for every large enough $j\in\N$, and recalling \eqref{lambdas-epsilon} we get
\begin{equation*}
                \big(\lambda\dist(y_j,\partial\Omega)^\beta\big)^2\Delta^N_p\phi(y_j)
                \geq
                o(1)
                \quad (j\to\infty).
\end{equation*}
Hence, taking limits as $j\to\infty$,
\begin{equation*}
                \big(\lambda\dist(x,\partial\Omega)^\beta\big)^2\Delta^N_p\phi(x)
                \geq
                0,
\end{equation*}
and thus $\Delta^N_p\phi(x)\geq 0$.
\end{proof}

In the next proposition we give a uniform boundary equicontinuity
estimate for $\{u_\varepsilon\}_{0<\varepsilon\leq 1}$. This
estimate is crucial to prove that the functions $\underline u$ and
$\overline u$ attach the right values near the boundary.

\begin{proposition}\label{THM-UNIFORM-EQUICONTINUITY}
Let $p\in [2, \infty)$,  $\Omega\subset\R^n$  a bounded domain
satisfying the uniform exterior cone condition (with constants
$\alpha$, $r$), $\gamma \in (0, \frac{1}{2}]$ as in \eqref{gamma}
and $f\in C(\partial \Omega )$.  Under the assumptions of
\Cref{MAIN-THM-CONVERGENCE}, let $u_\varepsilon\in \mathcal{K}_f$ be
the unique solution of \eqref{rhoeps}, for each $0<\varepsilon\leq
1$. Then for each $\eta>0$, there exists a constant $C>0$ depending
only on $\Omega$, $f$ and $\eta$ such that
\begin{equation}\label{regular-equicont}
                |u_\varepsilon(x)-f(\xi)|
                \leq
                C\gamma^{-2}|x-\xi|^\gamma+\eta,
\end{equation}
for every $x\in\overline\Omega$ and $\xi\in\partial\Omega$.
\end{proposition}

\begin{proof}

By uniform continuity, there exists small enough $\delta>0$ such
that $|f(\zeta)-f(\xi)|<\eta$ for every $\xi,\zeta\in\partial\Omega$
satisfying $|\zeta-\xi|<\delta$. Fix
$C=C_{f,\eta}=2\|f\|_\infty/\L(\delta)$, where $\L$ is the function
defined in \eqref{def L}. Then
\begin{equation*}
                |f(\zeta)-f(\xi)|
                \leq
                C\L(|\zeta-\xi|)+\eta
\end{equation*}
and from \eqref{lower-upper-control} it follows that
\begin{equation}\label{zeta xi eta}
                |f(\zeta)-f(\xi)|
                \leq
                Cw_\xi(\zeta)+\eta
\end{equation}
for every $\xi , \zeta\in\partial\Omega$, where $w_\xi$ is the
$\T_\rho$-barrier at $\xi$ defined in \eqref{Trho-barrier}. From the
fact that $w_{\xi}$ is a $\T_{\varepsilon}$-barrier at $\xi$
(\Cref{mainthmloc}) and the Comparison Principle (\Cref{THM-COMPARISON-PPLE}) it follows that
\begin{equation}
|u_{\varepsilon}(x) - f( \xi )| \leq Cw_{\xi}(x) + \eta
\end{equation}
for each $x\in \overline{\Omega}$ which, together with
\eqref{lower-upper-control}, implies \eqref{regular-equicont}.\qedhere
\end{proof}

\begin{remark}
If $u$ is the $p$-harmonic function in $\Omega$ with boundary data
$f$ then \eqref{regular-equicont} holds with $u_{\varepsilon}$
replaced by $u$. This follows from \eqref{zeta xi eta}, the fact
that $w_{\xi}$ is $p$-superharmonic  and the Comparison Principle
for the $p$-laplacian.

\end{remark}

\begin{proof}[Proof of \Cref{MAIN-THM-CONVERGENCE}]
We know that the semicontinuous functions $\underline u$ and
$\overline u$ defined in \eqref{pointwise-limits} satisfy
$\underline u\leq\overline u$ by construction. We claim that the theorem follows from the reverse inequality. Indeed, suppose that
$\underline u$ and $\overline u$ agree over the whole domain. This
allows to define a continuous function $u_0=\underline u=\overline
u\in \mathcal{K}_f$ as the pointwise limit
\begin{equation*}
                u_0(x)
                =
                \lim_{\varepsilon\to 0}u_\varepsilon(x)
\end{equation*}
for each $x\in\overline\Omega$.
Then, by \Cref{THM-p-SUBHARMONIC},
$u_0$ is both $p$-subharmonic and $p$-superharmonic, so $u_0$ is $p$-harmonic in $\Omega$
and $u_0\big|_{\partial\Omega}=f$.

Our strategy to show that $\underline u\geq \overline u$ relies on
the uniform equicontinuity estimate from
\Cref{THM-UNIFORM-EQUICONTINUITY} to show that $\underline u$ and
$\overline u$ take the right values near the boundary. The desired
inequality will then follow as a consequence of the comparison
principle for $p$-subharmonic and $p$-superharmonic functions.

Fix an arbitrary small $\eta>0$ and choose $C>0$ as in \Cref{THM-UNIFORM-EQUICONTINUITY}.
By the definition of $\overline u$
we have that
\begin{equation*}
                \overline u(x)-f(\xi)
                =
                \limsup_{y\to x,\ \varepsilon\to 0}\big(u_\varepsilon(y)-f(\xi)\big)
                \leq
                C\gamma^{-2}|x-\xi|^\gamma+\eta,
\end{equation*}
for $x\in\overline\Omega$, $\xi\in\partial\Omega$ and $\varepsilon>0$, where in the inequality we have used the estimate \eqref{regular-equicont}. Taking limits as $x\to\xi$ we get
\begin{equation*}
                \limsup_{x\to \xi}\big(\overline u(x)-f(\xi)\big)
                \leq
                \eta
\end{equation*}
for arbitrary small $\eta>0$.
Repeating an analogous argument for $\underline u$ we obtain
\begin{equation*}
                \limsup_{x\to \xi}\overline u(x)
                \leq
                f(\xi)
                \leq
                \liminf_{x\to \xi}\underline u(x)
\end{equation*}
for every $\xi\in\partial\Omega$. Since $\underline u$ and $\overline u$ are respectively $p$-superharmonic and $p$-subharmonic,
by the comparison principle \cite[Theorem 2.7]{JUU-LIN-MAN} we finally obtain that $\overline u\geq \underline u$ in $\Omega$.
\end{proof}

\appendix
\section{Auxiliary Lemmas}\label{appendix}

\begin{lemma}\label{aux-lemma}
Let $T\in[0,\frac{\pi}{2}]$ and $\gamma>0$. Then
\begin{multline*}
    \Big(\cos t+\sqrt{\sin^2T-\sin^2t}\Big)^\gamma\pm\Big(\cos t-\sqrt{\sin^2T-\sin^2t}\Big)^\gamma
    \\
    \leq
    (1+\sin T)^\gamma\pm(1-\sin T)^\gamma
\end{multline*}
whenever $|t|\leq T$.
\end{lemma}

\begin{proof}
Let $a\in[0,1]$ and define $\varphi_\pm:[a,1]\to\R$ by
\begin{equation*}
    \varphi_\pm(x)
    =
    \Big(x+\sqrt{x^2-a^2}\Big)^\gamma\pm\Big(x-\sqrt{x^2-a^2}\Big)^\gamma.
\end{equation*}
Direct computation shows that
\begin{equation*}
    \varphi'_\pm(x)
    =
    \frac{\gamma}{\sqrt{x^2-a^2}}\ \varphi_\mp (x)
    \geq
    0,
\end{equation*}
therefore $\varphi_\pm$ is positive and increasing in $[a,1]$. In particular, $\varphi_\pm(x)\leq\varphi_\pm(1)$ for every $x\in[a,1]$. Then the result follows by letting $a=\cos T$ and performing the change of variables $x=\cos t$.
\end{proof}

\begin{lemma}\label{<4}
Let $\gamma\in(0,1)$. Then
\begin{equation}\label{bound}
    \frac{\displaystyle\frac{x}{2} \big[(1+x)^\gamma-(1-x)^\gamma\big]}{\displaystyle 1-\frac{1}{2}\big[(1+x)^\gamma+(1-x)^\gamma\big]}
    \leq
    \frac{2}{1-\gamma}
\end{equation}
for all $x\in (0,1]$.
\end{lemma}

\begin{proof}
Let us recall the Taylor series of $f(x)=(1+x)^\gamma$:
\begin{equation}\label{series}
    (1+x)^\gamma
    =
    1+\sum_{k=1}^\infty\binom{\gamma}{k}x^k
\end{equation}
for $|x|\leq 1$,
where
\begin{equation*}
    \binom{\gamma}{k}
    =
    (-1)^{k-1}\frac{\gamma(1-\gamma)(2-\gamma)\cdots(k-1-\gamma)}{k!}
\end{equation*}
for each $k\in\N$. Observe that, since $\gamma\in(0,1)$, we have that
\begin{equation*}
    \binom{\gamma}{2k-1}>0
    \quad \text{ and } \quad
    \binom{\gamma}{2k}<0
    \quad \text{ for each $k\in\N$.}
\end{equation*}
We can rewrite the left-hand side in \eqref{bound} by replacing \eqref{series},
\begin{equation*}
    \frac{\displaystyle\frac{x}{2} \big[(1+x)^\gamma-(1-x)^\gamma\big]}{\displaystyle 1-\frac{1}{2}\big[(1+x)^\gamma+(1-x)^\gamma\big]}
    =
    \frac{\displaystyle\sum_{k=1}^\infty\binom{\gamma}{2k-1}x^{2k}}{\displaystyle-\sum_{k=1}^\infty\binom{\gamma}{2k}x^{2k}}.
\end{equation*}
Hence, \eqref{bound} follows from the fact that
\begin{equation*}
    \sum_{k=1}^\infty\brackets{\binom{\gamma}{2k-1}+\frac{2}{1-\gamma}\binom{\gamma}{2k}}x^{2k}
    \leq
    0
\end{equation*}
for every $x\in(0,1)$. In fact, every coefficient in the above series is nonpositive, that is
\begin{equation*}
    \binom{\gamma}{2k-1}+\frac{2}{1-\gamma}\binom{\gamma}{2k}
    =
    \binom{\gamma}{2k-1}\brackets{1+\frac{2}{1-\gamma}\cdot\frac{\gamma-2k+1}{2k}}
    \leq
    0
\end{equation*}
for every $k\in\N$.
\end{proof}


\end{document}